\documentclass[12pt]{article}

\newtheorem{theorem}{Theorem}
\newtheorem{corollary}{Corollary}
\newtheorem{lemma}{Lemma}
\newtheorem{remark}{Remark}

\newtheorem{definition}{Definition}

\newtheorem{proposition}{Proposition}

\usepackage{epsfig}
\usepackage{url}
\usepackage{amssymb}
\usepackage{amsmath}
\usepackage{amsfonts}

\def\qed{\begin{flushright} $\Box$ \end{flushright}}

\def\Dbar{\leavevmode\lower.6ex\hbox to 0pt{\hskip-.23ex \accent"16\hss}D}

\def\bR{{\mbox{\bf R}}}
\def\bZ{{\mbox{\bf Z}}}
\def\bC{{\mbox{\bf C}}}

\def\paf{{\mbox{\rm PAF}}}
\def\psd{{\mbox{\rm PSD}}}
\def\dft{{\mbox{\rm DFT}}}

\begin{document}

{\bf\LARGE
\begin{center}
Compression of Periodic Complementary Sequences and Applications\footnote{This research is
supported by NSERC grants.}
\end{center}
}

{\Large
\begin{center}
Dragomir {\v{Z}}. {\Dbar}okovi{\'c}\footnote{University of Waterloo,
Department of Pure Mathematics, Waterloo, Ontario, N2L 3G1, Canada
e-mail: \url{djokovic@math.uwaterloo.ca}}, Ilias S.
Kotsireas\footnote{Wilfrid Laurier University, Department of Physics
\& Computer Science, Waterloo, Ontario, N2L 3C5, Canada, e-mail:
\url{ikotsire@wlu.ca}}
\end{center}
}

\begin{abstract}
\noindent A collection of complex sequences of length $v$ is complementary if the sum of their periodic autocorrelation function values at all non-zero shifts is constant. For a complex sequence $A = [a_0,a_1,\ldots,a_{v-1}]$ of length $v = dm$ we define the $m$-compressed sequence $A^{(d)}$ of length $d$ whose terms are the sums $a_i + a_{i+d} + \cdots + a_{i+(m-1)d}$. We prove that the $m$-compression of a complementary collection of sequences is also complementary. The compression procedure can 
be used to simplify the construction of complementary 
$\{\pm1\}$-sequences of composite length. In particular, we construct several supplementary difference sets (SDS) $(v;r,s;\lambda)$ with $v$ even and $\lambda = (r+s)-v/2$, given here for the first time. 
There are $15$ normalized parameter sets $(v;r,s;\lambda)$ with $v\le50$ for which the existence question was open. We resolve 
all but one of these cases.
\end{abstract}

\section{Introduction}

We consider the ring $\bZ_v = \{ 0, 1, \ldots, v-1\}$ of integers modulo a positive integer $v$. Let $k_1,\ldots,k_t$ be positive integers and $\lambda$ an integer such that
\begin{equation} \label{par-lambda}
\lambda (v - 1) = \sum_{i=1}^t k_i (k_i - 1),
\end{equation} 
and let $X_1,\ldots,X_t$ be subsets of $\bZ_v$ such that 
\begin{equation} \label{kard-ki}
|X_i|=k_i, \quad i\in\{1,\ldots,t\}.
\end{equation} 

\begin{definition}
We say that $X_1,\ldots,X_t$ are {\em supplementary difference 
sets} (SDS) with parameters $(v;k_1,\ldots,k_t;\lambda)$, if for every nonzero element $c\in\bZ_v$ there are exactly $\lambda$ 
ordered pairs $(a,b)$ such that $a-b=c \pmod{v}$ and 
$\{a,b\}\subseteq X_i$ for some $i\in\{1,2,\ldots,t\}$.
\end{definition}

These SDS are defined over the cyclic group $\bZ_v$. More generally SDS can be defined over any finite abelian group, 
and there are also further generalizations where the group
may be any finite group. However, in this paper we shall 
consider only the cyclic case.

In the context of an SDS, say $X_1,\ldots,X_t$, with parameters $(v;k_1,\ldots,k_t;\lambda)$, we refer to the subsets $X_i$ as the {\em base blocks} and we introduce an additional parameter, $n$, defined by: 
\begin{equation} \label{par-n}
n = k_1 + \cdots + k_t - \lambda. 
\end{equation} 

The SDS formalism is a generalization of various well-known and studied designs. 

SDS consisting of just one base block, the case $t=1$, are called cyclic difference sets and are denoted by 
$(v;k;\lambda)$. The reader is referred to the classic book \cite{Baumert:1971}, the more recent book \cite{Stinson:book:2004} and the recent survey \cite{Handbook:DifferenceSets}. The list of $(v;k;\lambda)$ cyclic difference sets with $v \leq 50$ in a normal form can be found in section $3$ of \cite{Djokovic:AnnComb:2011}.

The SDS with two base blocks, the case $t=2$, covers several 
interesting designs. First, the circulant D-optimal designs 
\cite{KO:2007,DK:JCD:2012} corresponding to the case where $v$ is odd and $n=(v-1)/2$. Second, the binary complementary pairs (periodic analogs of complementary Golay pairs) \cite{Yang:1971} 
correspond to the case $v=2n$. 
Third, when the two base blocks have the same size the SDS 
can be used to construct some BIBD (balanced incomplete 
block designs) \cite{MR:HCD2007,Stinson:book:2004}.  
We shall denote the SDS with two base blocks of size $r$ and 
$s$ by $(v;r,s;\lambda)$. Upon imposing the normalization condition $\frac{v}{2} \geq r \geq s \geq 2$, it can be seen that for $v \leq 50$ there are $227$ feasible parameter sets, see \cite{Djokovic:AnnComb:2011}. These parameter sets are 
those that satisfy the conditions (\ref{par-lambda}) and
(\ref{kard-ki}) with $k_1=r$, $k_2=s$. There is a non-existence result for SDS $(v;r,s;\lambda)$ obtained in \cite{Arasu:Xiang:DCC:1992}, which can be used to eliminate some 
of these 227 feasible parameter sets. In \cite{MDV:JCD:2004} the authors construct $30$ new $(v;r,s;\lambda)$ and show that an additional $19$ $(v;r,s;\lambda)$ do not exist. In \cite{Djokovic:AnnComb:2011} the author constructs $8$ new $(v;r,s;\lambda)$. We list here the remaining $15$ undecided cases from \cite{Djokovic:AnnComb:2011}. They all have $v$ in the interval $40 < v \leq 50$.

\begin{verbatim}
(41;15,6;6)     n=15    (43;9,4;2)      n=11    (44;19,2;8)     n=13
(45;18,2;7)     n=13    (46;21,6;10)    n=17    (47;9,5;2)      n=12
(47;12,3;3)     n=12    (47;14,2;4)     n=12    (47;15,5;5)     n=15
(48;14,3;4)     n=13    (49;10,3;2)     n=11    (49;21,4;9)     n=16
(50;8,7;2)      n=13    (50;20,4;8)     n=16    (50;22,21;18)   n=25
\end{verbatim}

In this paper (see Theorem \ref{thm:v<=50}) we decide the existence of all of them except for the case $(49;21,4;9)$. 
The main tools that we use for that purpose are computational techniques based on the Power Spectral Density ({\psd}) test as well as the method of SDS compression. For the PSD test see section \ref{sec:Complementarity}. The compression method is studied systematically and in full generality (for arbitrary complex sequences and in particular for SDS) in section \ref{sec:Compression}. We note that techniques that essentially amount to compression have been used previously \cite{Cohn:1992,KKNK:1994} in the context of D-optimal designs.

Apart from the cases where there exist Golay complementary 
pairs of length $v$, the binary complementary pairs are known 
to exist only for $v=34$ \cite{Djokovic:DesCodes:1998}
and $v=50$ \cite{Djokovic:AnnComb:2011,KK:2008}. In this paper we 
construct such pairs of length $v=50$ and $v=58$ (see the section \ref{sec:results}). The example for $v=50$ has 
different parameters, namely $(50;22,21;18)$, from the previously 
constructed examples with parameters $(50;25,20;20)$.
The first unknown case is now for $v=68$.

\section{The group ring approach to SDS}

In the study of SDS it is often convenient to use the group ring of the additive group $\bZ_v$ with coefficients in $\bC$, the field of complex numbers. This 
group ring can be identified with the quotient ring 
$\bC[x]/(x^v-1)$, where $\bC[x]$ is the polynomial ring over 
$\bC$ in a single indeterminate $x$ and $(x^v-1)$ is the ideal 
generated by the polynomial $x^v-1$. By abuse of notation, we 
shall consider $x$ also as an element of this group ring, in 
which case we have $x^v=1$. 
Then $\bZ_v$ is embedded in this group ring via
the map which sends $i\to x^i$, $i\in\bZ_v$.
Note that because $x^v=1$ there is a ring homomorphism
(evaluation at 1)
$\bC[x]/(x^v-1)\to\bC$ which sends $x\to1$.

To any subset $X\subseteq\bZ_v$ we shall associate an element
of the group ring, namely $X(x):=\sum_{i\in X} x^i$.
By abuse of notation, we shall denote this element also by $X$, except that in the case $X=\bZ_v$ we set
$$ T=T(x)=1+x+\cdots+x^{v-1}. $$
It will be clear from the context which meaning of $X$ is used.

The group ring has an involution which sends each complex number $c$ to its complex conjugate $\overline{c}$ and sends $x$ to its
inverse $x^{-1}=x^{v-1}$. We shall denote this involution
by an asterisk, e.g., we have $x^*=x^{-1}$ and $T^*=T$.
For any element $X$ of the group ring we define its
{\em norm} to be the element $N(X)=XX^*$.
We also define the function $N_X:\bZ_v\to\bZ$ by declaring that
$N_X(s)$ is the coefficient of $x^s$ in $N(X)$, i.e., we have
\begin{equation*} 
N(X)=\sum_{i=0}^{v-1} N_X(i)x^i.
\end{equation*}

To any complex sequence of length $v$, say
$A=[a_0,a_1,\ldots,a_{v-1}]$ we assign the element
$A(x)=\sum_i a_i x^i$ of the group ring. Obviously,
the map sending $A$ to $A(x)$ is injective. 
To simplify notation, we shall write $A$ instead of $A(x)$ 
if no confusion will arise.
By evaluating $A(x)$ at 1, we obtain $A(1)=\sum_i a_i$.
Similarly, we have $N(A)(1)=|A(1)|^2$.

We point out that the SDS property is equivalent to an identity in the group ring. This is stated in the next lemma whose 
proof is straightforward.
\begin{lemma}
Let $X_1,\ldots,X_t$ be subsets of $\bZ_v$ and assume that Eqs. (\ref{par-lambda}) and (\ref{kard-ki}) hold. Then these subsets are the base blocks of an SDS with parameters 
$(v;k_1,\ldots,k_t;\lambda)$ if and only if 
\begin{equation}\sum_{i=1}^t N(X_i)=n+\lambda T,
\label{jed:zbir-nor}
\end{equation}
where $n$ is defined by Eq. (\ref{par-n}).
\end{lemma}

\section{Power spectral density}

Many known SDS have been constructed by a computer search. 
In such searches it is often convenient to replace a subset 
$X\subseteq\bZ_v$ with the $\{\pm1\}$-sequence 
$A=[a_{0},a_{1},\ldots,a_{v-1}]$ where $a_i=-1$ if and only if 
$i\in X$. 
We shall refer to $A$ as the {\em associated sequence} of $X$.
In that case we have 
\begin{equation} \label{jed:A-X}
A(x)=T(x)-2X(x). 
\end{equation}

Let $A = [a_0,a_1,\ldots,a_{v-1}]$ be an arbitrary complex 
sequences of length $v$, and let 
$\omega = e^{2\pi i/v}$ be a primitive $v$-th root of unity. 

The {\em Discrete Fourier Transform} ({\dft}) of the sequence $A$ is the function $\bZ_v\to\bC$ defined by the formula
$$ \dft_A(s)=\sum_{j=0}^{v-1} a_j\omega^{js}. $$

We often identify a function on $\bZ_v$ with the sequence of its values, e.g., we write
$$
{\dft}_A = B = [b_0,b_1,\ldots,b_{v-1}], \text{ where } 
b_j =\dft_A(j).
$$
This convention is used throughout the paper.

The {\em Power Spectral Density} ({\psd}) of the same sequence 
$A$ is the function $\bZ_v\to\bR$ defined by the formula
$$ \psd_A(s)=\left|\dft_A(s)\right|^2. $$

The {\em Periodic Autocorrelation Function} ({\paf}) of $A$ is defined as
$$ {\paf}_A(s) = \sum_{j=0}^{v-1} a_{j+s} \bar{a}_j. $$
We shall refer to the argument $s$ as the {\em shift} variable.

It is straightforward to verify that for any complex sequence 
$A=[a_0,a_1,\ldots,a_{v-1}]$ and the corresponding element 
$A=A(x)=a_0+a_1x+\cdots+a_{v-1}x^{v-1}$ of the group ring we have 
\begin{equation} \label{jed:N=PAF}
N_A=\paf_A. 
\end{equation}
In general we have $N_A(-s)=\overline{N_A(s)}$, and for a 
real sequence $A$ we have $N_A(-s)=N_A(s)$. 
Another important fact is the following classical theorem 
(see e.g. \cite{Ricker:book:2003}).

\begin{theorem}[Wiener--Khinchin] 
For any complex sequence $A=[a_0,a_1,\ldots,a_{v-1}]$, or 
equivalently any element $A=a_0+a_1x+\cdots+a_{v-1}x^{v-1}$ 
of the group ring, we have 
\begin{equation*} 
    {\psd}_A = {\dft}({\paf}_A).
\end{equation*}
\end{theorem}

By using Eq. (\ref{jed:N=PAF}) this can be rewritten as follows:
\begin{equation} \label{jed:W-K}
    {\psd}_A(s) = \sum_{r=0}^{v-1} N_A(r)\omega^{rs}.
\end{equation}
If the sequence $A$ is real, then also
\begin{equation} \label{jed:W-K-realna}
    {\psd}_A(s) = \sum_{j=0}^{v-1} N_A(j)\cos\frac{2\pi js}{v}.
\end{equation}

\section{Complementary sequences} \label{sec:Complementarity}

The associated sequences of the base blocks of an SDS have 
an important property known as complementarity. Let us 
begin with the general definition of this property.

\begin{definition}
Let $A_1,\ldots,A_t$ be complex sequences of length $v$. We say that these sequences are {\em complementary} if
$$
\sum_{i=1}^t {\paf}_{A_i} = [\alpha_0, \underbrace{\alpha, \ldots, \alpha}_{v-1 ~\mbox{\rm terms}}]
$$
for some $\alpha_0$ and $\alpha$ (the \paf-constants).
\end{definition}

In the next theorem we show that a finite collection of complex sequences has constant sum of PAF values at nonzero shifts if 
and only if the same is true for the PSD values, and we find 
simple formulae expressing the PSD-constants in terms of the 
PAF-constants. (This theorem is a generalization of 
\cite[Theorem 2]{FGS:2001}.)

\begin{theorem}
Let $A_1,\ldots,A_t$ be complex sequences of length $v$. These sequences are complementary, i.e., 
\begin{equation}
    \sum_{i=1}^t {\paf}_{A_i} = [\alpha_0, \underbrace{\alpha,\ldots,\alpha}_{v-1 ~\mbox{\rm terms}}]
    \label{PAF_constants}
\end{equation}
if and only if
\begin{equation}
    \sum_{i=1}^t {\psd}_{A_i} = [\beta_0, \underbrace{\beta,\ldots,\beta}_{v-1 ~\mbox{\rm terms}}].
    \label{PSD_constants}
\end{equation}
The $\paf$-constants $\alpha_0$ and $\alpha$ and the 
$\psd$-constants $\beta_0$ and $\beta$ are related as follows: 
\begin{equation}
\beta_0 = \alpha_0+(v-1)\alpha,\quad \beta=\alpha_0-\alpha.
    \label{PSD_to_PAF_relationships}
\end{equation}
\label{Theorem:PAF_PSD_relationships}
\end{theorem}

\begin{proof}
Note that (\ref{PAF_constants}) can be written as:
$$
\sum_{i=1}^t {\paf}(A_i)=(\alpha_0-\alpha) \cdot [1,0,\ldots,0]
+ \alpha \cdot [1,\ldots, 1].
$$
By applying the DFT and by using the Wiener--Khinchin theorem and the fact that {\dft} is a linear operator, this implies that
$$
\begin{array}{rcl}
\displaystyle\sum_{i=1}^t {\psd}(A_i)&=&(\alpha_0-\alpha)\cdot 
{\dft}[1,0,\ldots,0]+\alpha \cdot {\dft}[1,\ldots, 1] \\
                          & = & (\alpha_0 - \alpha) \cdot [1,\ldots,1] + \alpha \cdot [v,0,\ldots,0] \\
                          & = & [\beta_0,\beta,\ldots,\beta] \\
\end{array}
$$
Hence (\ref{PSD_constants}) holds as well as (\ref{PSD_to_PAF_relationships}). 
Conversely, assume that (\ref{PSD_constants}) holds.
By applying the inverse {\dft} we obtain
$$
\begin{array}{rcl}
    \displaystyle\sum_{i=1}^t {\paf}(A_i)&=&(\beta_0-\beta)\cdot {\dft}^{-1}[1,0,\ldots,0]+\beta\cdot{\dft}^{-1}[1,\ldots,1] \\
                          &  = & \displaystyle\frac{\beta_0 - \beta}{v}[1,\ldots,1] + \beta [1,0,\ldots,0]  \\
                          & = & [\alpha_0, \alpha, \ldots, \alpha] \\
\end{array}
$$
\qed
\end{proof}

By solving the equations (\ref{PSD_to_PAF_relationships}) we 
obtain that
\begin{equation*}
    \alpha_0 = \displaystyle\frac{\beta_0 - \beta}{v}+ \beta, 
    \quad \alpha = \displaystyle\frac{\beta_0 - \beta}{v}.
\end{equation*}

Note that the PSD values are always nonnegative. Hence, if 
$A_1,\ldots,A_t$ are complementary complex sequences of length 
$v$ with the PSD-constants $\beta_0$ and $\beta$, then for 
$r=1,\ldots,t$ we have
\begin{equation} \label{jed:psd-test}
\psd_{A_r}(s)\le\beta, \quad s=1,2,\ldots,v-1.
\end{equation}
We shall refer to this inequality as the {\em PSD-test}.

In the following proposition we show that the associated sequences of the base blocks of an SDS are complementary sequences and we compute the PAF and PSD-constants. In a special 
case these constants were computed in 
\cite[Example 3]{FGS:2001}.

\begin{proposition} \label{SDS-kompl}
Let $A_1,\ldots,A_t$ be the $\{\pm1\}$-sequences associated to the base blocks $X_1,\ldots,X_t$ of an SDS with parameters 
$(v;k_1,\ldots,k_t;\lambda)$. Then  
\begin{equation} \label{jed:zbir-nor-A}
\sum_{i=1}^t N(A_i)=4n+(tv-4n)T.
\end{equation}
Moreover, the sequences $A_1,\ldots,A_t$ are complementary 
with $\paf$-constants 
\begin{equation} \label{jed:SDS-paf}
\alpha_0=tv, \quad \alpha=tv-4n,
\end{equation}
and $\psd$-constants 
\begin{equation} \label{jed:SDS-psd}
\beta_0=tv, \quad \beta=4n.
\end{equation}
\end{proposition}

\begin{proof}
Note that $x^i T=T$ for each $i$, and so we have $T^2=vT$ and $TX_i=TX_i^*=k_iT$. By Eq. (\ref{jed:A-X}) we have 
$N(A_i)=(T-2X_i)(T-2X_i^\star)$ which gives 
\begin{equation} \label{jed:NA-NX}
N(A_i)=(v-4k_i)T +4N(X_i).
\end{equation}
Summing over all $i = 1,\ldots,t$ and by using Eq. 
(\ref{jed:zbir-nor}), we obtain Eq. (\ref{jed:zbir-nor-A}).
Using Eq. (\ref{jed:N=PAF}), we have
$$ N(A_i)=\sum_{j=0}^{v-1} \paf_{A_i}(j)x^j. $$
By summing these equations over $i=1,\ldots,t$ and by using
(\ref{jed:zbir-nor-A}), we deduce that
\begin{equation*}
\sum_{i=1}^t \paf_{A_i}(j) =
\left\{
\begin{array}{cc}
tv-4n, & j \neq 0, \\
tv,    & j = 0. \\
\end{array}
\right.
\end{equation*}
Hence, the equations (\ref{jed:SDS-paf}) hold, and by 
using Theorem \ref{Theorem:PAF_PSD_relationships} we deduce that 
the equations (\ref{jed:SDS-psd}) also hold. 
\qed
\end{proof}

When searching for SDS, we can discard the trial base block $X_r$ if the associated sequence $A_r$ fails the PSD-test: 
$\psd_{A_r}(s)>\beta$ for some $s\ne0$. In this case we can 
restate the PSD-test as follows.

\begin{lemma} \label{le:Test}
Under the hypotheses of Proposition \ref{SDS-kompl}, for 
$r=1,\ldots,t$ we have
\begin{equation} \label{jed:Test}
\sum_{j=1}^{v-1} N_{X_r}(j)\cos\frac{2\pi js}{v}\le n-k_r, 
\quad s\in\{1,\ldots,v-1\}.
\end{equation}
\end{lemma}

\begin{proof}
Recall that the DFT of a constant function vanishes for all 
shifts $s$ except for $s=0$. Hence, by using Eqs. (\ref{jed:N=PAF}) and (\ref{jed:NA-NX}), we obtain that  
$\dft(\paf_{A_r})(s)=4\cdot\dft(N_{X_r})(s)$. Thus, we must have  
$$ 
4\left( k_r+\sum_{j=1}^{v-1}
N_{X_r}(j)\cos\frac{2\pi js}{v} \right)\le\beta=4n, \quad s\ne0,
$$
i.e., (\ref{jed:Test}) holds.
\end{proof}

\section{Compression of sequences} \label{sec:Compression}

If we have a collection of complementary sequences of length 
$v=dm$, then we can compress them to obtain complementary 
sequences of length $d$. We refer to the ratio $v/d=m$ as 
the {\em compression factor}. Here is the precise definition.

\begin{definition}
Let $A = [a_0,a_1,\ldots,a_{v-1}]$ be a complex sequence of length $v = dm$ and set 
\begin{equation} \label{koef-kompr}
a_j^{(d)}=a_j+a_{j+d}+\ldots+a_{j+(m-1)d}, \quad 
j=0,\ldots,d-1. 
\end{equation}
Then we say that the sequence 
$A^{(d)} = [a_0^{(d)},a_1^{(d)},\ldots,a_{d-1}^{(d)}]$ 
is the {\em $m$-compression} of $A$.
\end{definition}

Let us now show that the complementarity is preserved 
by the compression process. At the same time we shall compute the 
PAF and PSD-constants of the compressed sequences.
\begin{theorem}
Let $A_i = [a_{i0},a_{i1},\ldots,a_{i,v-1}]$, $i=1,\ldots,t$, 
be complementary complex sequences of length $v=dm$ with the 
{\paf}-constants $\alpha_0,\alpha$. Then the corresponding $m$-compressed sequences $A_i^{(d)}=[a_{i0}^{(d)},\ldots,a_{i,d-1}^{(d)}]$, 
$i=1,\ldots,t$, are complementary with $\paf$-constants:

\begin{equation} \label{paf-konst-komp}
\alpha_0^{(d)}=\alpha_0+(m-1)\alpha, \quad 
\alpha^{(d)}=m\alpha.
\end{equation}
The \psd-constants of the original and compressed sequences are 
the same.
\label{Theorem:Main_Theorem}
\end{theorem}

\begin{proof}
By using (\ref{koef-kompr}) we compute ${\paf}_{A_i^{(d)}}(k)$ for arbitrary shift $k$:
$$
\begin{array}{rcl}
    {\paf}_{A_i^{(d)}}(k)  & = & \displaystyle\sum_{j=0}^{d-1} a_{i,j+k}^{(d)} \overline{a_{i,j}^{(d)}}    \\
                        & = & \displaystyle\sum_{r,s = 0}^{m-1} \displaystyle\sum_{j = 0}^{d-1} a_{i,j+k+rd} \overline{a_{i,j+sd}}. \\
\end{array}
$$
By applying the substitution $r \rightarrow r+s$, we obtain
$$
\begin{array}{rcl}
    {\paf}_{A_i^{(d)}}(k)  & = & \displaystyle\sum_{r,s = 0}^{m-1} \displaystyle\sum_{j = 0}^{d-1} a_{i,j+sd+k+rd} \overline{a_{i,j+sd}} \\
                        & = & \displaystyle\sum_{r = 0}^{m-1} {\paf}_{A_i}(k+rd) \\
\end{array}
$$
By summing over $i$ from $1$ to $t$, we obtain
$$
\begin{array}{rcl}
    \displaystyle\sum_{i=1}^t {\paf}_{A_i^{(d)}}(k) & = & \displaystyle\sum_{r = 0}^{m-1} \displaystyle\sum_{i=1}^{t} {\paf}_{A_i}(k+rd). \\
\end{array}
$$
Since ${\paf}_{A_i}(k+rd)$ is equal to $\alpha_0$ if $k=r=0$ and is equal to $\alpha$ otherwise, we conclude that the sequences 
$A_1^{(d)},\ldots,A_t^{(d)}$ are complementary and that their 
PAF-constants are given by Eqs. (\ref{paf-konst-komp}). 
The assertion about the PSD-constants follows from Theorem 
\ref{Theorem:PAF_PSD_relationships}.
\qed
\end{proof}

By applying Theorem \ref{Theorem:Main_Theorem} directly to SDS with parameters $(v;k_1,\ldots,k_t;\lambda)$, we obtain the following result.

\begin{theorem}
Let $A_1,\ldots,A_t$ be the sequences associated to 
the base blocks of an SDS with parameters 
$(v;k_1,\ldots,k_t;\lambda)$. Suppose that $v=dm$ and let 
$A_1^{(d)},\ldots,A_t^{(d)}$ be the corresponding $m$-compressed sequences. Then the PAF-constants of the compressed sequences 
are given by 
\begin{equation}
\alpha_0^{(d)}=m(tv-4n)+4n,\quad \alpha^{(d)}=m(tv-4n),
\label{Equation_General_Constraint}
\end{equation}
where $n$ is defined by Eq. (\ref{par-n}).
\label{Theorem:Generalized_Constraint}
\end{theorem}

\begin{proof}
By Proposition \ref{SDS-kompl}, the sequences $A_1,\ldots,A_t$ 
are complementary with {\paf}-constants given by Eqs. 
(\ref{jed:SDS-paf}). It remains to apply Theorem 
\ref{Theorem:Main_Theorem}.
\qed
\end{proof}

\begin{remark}
Theorem \ref{Theorem:Generalized_Constraint} generalizes 
\cite[Theorem 1]{DK:JCD:2012}, which deals with the special 
case: SDS $(v;r,s;\lambda)$ with $n=r+s-\lambda=(v-1)/2$. 
(These SDS are used to construct circulant D-optimal matrices.) \end{remark}

\begin{remark}
Theorem \ref{Theorem:Generalized_Constraint} has two useful corollaries for
\begin{itemize}
\item $m = 2$, in which case $|a_{ij}^{(d)}| \in \{ 0, 2 \}$,
\item $m = 3$, in which case $|a_{ij}^{(d)}| \in \{ 1, 3 \}$. \end{itemize}
\end{remark}
An easy counting argument allows us to compute the number of 
$a_{ij}^{(d)}$ that assume a specific absolute value.

\begin{corollary} 
With the hypotheses and notation of Theorem \ref{Theorem:Generalized_Constraint}, let $m = 2$ and denote by $\nu_0, \nu_2$ the number of $a_{ij}^{(d)}$ terms in all 
$2$-compressed sequences $A_i^{(d)}$ that have absolute value equal to $0,2$ respectively. Then $\nu_0=n$ and $\nu_2=td-n$.
\label{Corollary:2-compression}
\end{corollary}

\begin{proof}
Since $m=2$, from Eq. (\ref{Equation_General_Constraint}) we have $\alpha_0^{(d)}=4(td-n)$. Therefore 
$0\cdot\nu_0+4\cdot\nu_2=4(td-n)$, i.e., $\nu_2=td-n$. As 
$\nu_0+\nu_2=td$, we have $\nu_0=n$.
\qed
\end{proof}

\begin{corollary}
With the hypotheses and notation of Theorem \ref{Theorem:Generalized_Constraint}, let $m = 3$ and denote by $\nu_1,\nu_3$ the number of $a_{ij}^{(d)}$ terms in the $3$-compressed sequences $A_i^{(d)}$ that have absolute value equal to $1,3$ respectively. Then $\nu_1=n$ and $\nu_3=td-n$.
\label{Corollary:3-compression}
\end{corollary}

\begin{proof}
Since $m=3$, from Eq. (\ref{Equation_General_Constraint}) we have $\alpha_0^{(d)}=9td-8n$. Therefore 
$1\cdot\nu_1+9\cdot\nu_3=9td-8n$. As $\nu_1+\nu_3=td$, we have 
$\nu_1=n$ and $\nu_3=td-n$.
\qed
\end{proof}

\section{Computational results for SDS with two base blocks} 
\label{sec:results}

\subsection{Necklaces, bracelets and charmed bracelets}

When compressing the pair $A,B$ of periodic complementary 
$\{\pm1\}$ sequences of length $v=dm$, with compression 
factor $m$, we obtain the pair $A^{(d)},B^{(d)}$ of 
periodic complementary sequences of length $d$ whose elements
are integers belonging to the set $\{m,m-2,\ldots,2-m,-m\}$.
The first equation of (\ref{jed:SDS-psd}) shows that the
sum of squares of the elements of $A^{(d)}$ and $B^{(d)}$ 
together is equal to $2v$. In our applications we only used
the compression factors $m=2$ and $m=3$. In these two cases the
Corollaries \ref{Corollary:2-compression} and 
\ref{Corollary:3-compression} determine the total number of 
squares $0^2,2^2$ and $1^2,3^2$, respectively. There are 
several ways to distribute these squares over the sequences
$A^{(d)}$ and $B^{(d)}$, and we treat each of them separately.

We shall use combinatorial objects known as necklaces, 
bracelets and charmed bracelets \cite{Sawada:2001,Sawada:2003}
Roughly speaking the {\em necklaces} are cyclic arrangements of $v$ objects where we do not distinguish the arrangements obtained by cyclic shifts. If we enlarge the equivalence classes by allowing also that the arrangement be reversed, then we refer to these equivalence classes as {\em bracelets}. We can further enlarge these equivalence classes by moving, for each $i$, the object in position $i$ to the position $si \pmod{v}$ where $s$ is a fixed integer relatively prime to $v$. We shall refer to this operation as a {\em multiplication} by $s$. Note that the multiplication by $s=-1$ is the same as the reversal. We refer to these enlarged equivalence classes as {\em charmed racelets}. We emphasize that the multiplication operation does change the PAF of a sequence, but when performed simultaneously on a pair of complementary sequences it preserves the complementarity 
property.

The sequences $A$ and $B$, as well as the compressed sequences $A^{(d)}$ and $B^{(d)}$, can be viewed as being cyclic 
arrangements of length $v$ and $d$, respectively. The cyclic shifts and the reversal do not change the periodic autocorrelation of the sequence, and so they can be applied separately on $A$ and $B$ without destroying the complementarity property. However, the multiplication by $s\ne-1 \pmod{v}$ has to be performed simultaneously, with the same $s$, on both $A$ and $B$ as otherwise the complementarity property may be destroyed. 

In our searches using the compression method we first construct 
the compressed sequences $A^{(d)}$ and $B^{(d)}$. They must 
have the specified row sums and pass the PSD test. We choose 
the first sequence to be a representative of a charmed 
bracelet while the second sequence can be chosen only as a 
representative of an ordinary bracelet. Then we proceed to 
select the pairs $A^{(d)},B^{(d)}$ which are complementary. If there are no such pairs this means that the SDS that we are looking for do not exist. Otherwise we examine each of the complementary pairs $A^{(d)},B^{(d)}$ and try to lift them to obtain a complementary pair $A,B$.

\subsection{The range $v \leq 50$} \label{sub:range-1}

We shall prove the following theorem resolving the existence of 
all but one of the $15$ open cases from the list in the Introduction. The only remaining undecided case is now $(49;21,4;9)$. 

\begin{theorem} $ $ \label{thm:v<=50}  \\
\begin{itemize}
\item There do not exist SDS with following parameters:
$(41;15,6;6)$
$(43;9,4;2)$
$(44;19,2;8)$
$(45;18,2;7)$
$(46;21,6;10)$
$(47;9,5;2)$
$(47;12,3;3)$
$(47;14,2;4)$
$(47;15,5;5)$
$(48;14,3;4)$
$(49;10,3;2)$
$(50;8,7;2)$
$(50;20,4;8)$
\item There exist SDS with the following parameters:
$(50;22,21;18)$
\end{itemize}
\end{theorem}

{\bf Proof} We proceed by examining the $14$ $(v;r,s;\lambda)$ cases one by one.
\begin{enumerate}

\item For the case $(41;15,6;6)$ the application of the {\psd} test with the constant value $60$
gave $1040$ normalized candidate A-sequences and $13104$ normalized candidate B-sequences.
Subsequently, there was no match found between these two sets of candidate sequences.

\item For the case $(43;9,4;2)$ there are no A-sequences that pass the {\psd} test with the constant value $44$,
therefore we conclude immediately that such SDS do not exist.

\item For the case $(44;19,2;8)$ there are no A-sequences that pass the {\psd} test with the constant value $52$,
therefore we conclude immediately that such SDS do not exist.

\item $(45;18,2;7)$ there are no A-sequences that pass the {\psd} test with the constant value $52$,
therefore we conclude immediately that such SDS do not exist.

\item For the case $(46;21,6;10)$ we used $2$-compression to prove that there do not exist such SDS.
Here are the details of this computation. First note that in this case we have $n = r+s-\lambda = 21+6-10 = 17$.
Taking $m = 2, d = 23$ in Theorem \ref{Theorem:Generalized_Constraint}, we see that we need to find two sequences of length $23$ with elements from $\{-2,0,+2\}$  such that their {\psd} values add up to $4n = 4 \cdot 17 = 68$. Since we know from corollary \ref{Corollary:2-compression}
that the total number of $0$ elements in the two sequences is equal to $\nu_0 = 17$ and that the total number of $\pm 2$ elements in the two sequences is equal to $\nu_2 = 29$, we deduce that there are four cases to consider, using ordinary and charmed bracelets \cite{Sawada:2001}.
For each one of these four cases: (1) we compute
the charmed bracelets that give $2$-compressed sequences of length $23$ all of whose {\psd} values are smaller than $68$
and (2) we compute the ordinary bracelets that give $2$-compressed sequences of length $23$ all of whose {\psd} values are smaller than $68$. We summarize the results in the following tables.

\begin{tabular}{c|c|c}
\multicolumn{3}{c}{A-sequences} \\
Case & charmed bracelets & charmed bracelets passing {\psd} \\
\hline
1 & 2,116,296 & 85   \\
2 & 475,020   & 2,009 \\
3 & 54,264    & 4,552 \\
4 & 3,015     & 1,442 \\
\end{tabular}

\begin{tabular}{c|c|c}
\multicolumn{3}{c}{B-sequences} \\
Case & ordinary bracelets & ordinary bracelets passing {\psd} \\
\hline
1 & 2,277 & 1,749   \\
2 & 3,685   & 1,419 \\
3 & 1,210    &  22 \\
4 &   44     &   0 \\
\end{tabular}

Notice that case 4 can be eliminated at this stage, since there are no ordinary bracelets that pass the {\psd} test.
The next step is for each of the remaining three cases to look for pairs of sequences that have constant {\paf} and
we summarize the results in the following table (where the symbol $\rightarrow$ indicates removal of sequences with duplicate PAF)

\begin{tabular}{c|c|c|c}
Case & A-sequences & B-sequences & \# of pairs\\
\hline
1 & 85 $\rightarrow$ 84       & 1,749 $\rightarrow$ 1,716  & 39 \\
2 & 2,009 $\rightarrow$ 1,970 & 1,419 $\rightarrow$ 1,419  & 34 \\
3 & 4,552 $\rightarrow$ 4,497 &  22   $\rightarrow$ 22     & 0 \\
\end{tabular}

Notice that case 3 can be eliminated at this stage, since there are no pairs of matching sequences at all.
Subsequently we used the $39+34 = 73$ matching pairs of sequences of length $23$ with elements from $\{-2,0,+2\}$ to
construct all the corresponding pairs of uncompressed sequences of length $46$ and check whether they form an SDS with
the required parameters. We did not find any such SDS.

\item For the case $(47;9,5;2)$ there are no A-sequences that pass the {\psd} test with the constant value $48$,
therefore we conclude immediately that such SDS do not exist.

\item For the case $(47;12,3;3)$ there are no A-sequences that pass the {\psd} test with the constant value $48$,
therefore we conclude immediately that such SDS do not exist.

\item For the case $(47;14,2;4)$ there are no A-sequences that pass the {\psd} test with the constant value $48$,
therefore we conclude immediately that such SDS do not exist.

\item For the case $(47;15,5;5)$ there are no A-sequences that pass the {\psd} test with the constant value $60$,
therefore we conclude immediately that such SDS do not exist.

\item For the case $(48;14,3;4)$ there are no A-sequences that pass the {\psd} test with the constant value $52$,
therefore we conclude immediately that such SDS do not exist.

\item For the case $(49;10,3;2)$ there are no A-sequences that pass the {\psd} test with the constant value $44$,
therefore we conclude immediately that such SDS do not exist.

\item For the case $(50;8,7;2)$ the application of the {\psd} test with the constant value $52$
gave $1130$ normalized candidate A-sequences and $2910$ normalized candidate B-sequences.
Subsequently, there was no match found between these two sets of candidate sequences.

\item For the case $(50;20,4;8)$ we used $2$-compression to prove that there do not exist such SDS.

\item We give four non-equivalent examples of SDS 
$(50;22,21;18)$:
$$
\begin{array}{l}
  \{ 0,1,2,3,6,7,9,13,14,16,18,20,22,23,26,27,30,35,37,41,45,46 \} \\
  \{ 0,1,2,3,4,5,6,8,11,12,14,17,20,22,29,30,32,37,38,39,42 \}
\end{array}
$$

$$
\begin{array}{l}
 \{ 0,1,2,3,4,6,7,8,9,14,16,18,20,21,25,31,32,35,36,42,44,45 \} \\
 \{ 0,1,2,4,5,8,9,10,12,14,18,21,23,24,27,29,32,34,35,39,42 \}
\end{array}
$$

$$
\begin{array}{l}
 \{ 0,1,2,3,5,8,9,11,14,15,19,21,24,25,29,30,32,36,38,39,41,43 \} \\
 \{ 0,1,3,5,6,7,8,9,10,13,16,18,20,21,24,25,31,32,33,37,41 \}
\end{array}
$$

$$
\begin{array}{l}
  \{ 0,2,3,4,6,9,10,12,13,17,19,20,24,25,28,29,30,33,38,39,41,47 \} \\
  \{ 0,1,3,5,6,7,8,10,12,13,14,17,20,22,24,28,32,37,38,39,40 \}
\end{array}
$$

All four examples are in the canonical form defined in \cite{Djokovic:AnnComb:2011} and since they are different, this implies that they are non-equivalent. Note that since $2v-4n = 2 \cdot 50 - 4 \cdot 25 = 0$, these SDS give rise to $\{ \pm 1 \}$ sequences of length $50$ with {\paf} zero, because of 
(\ref{jed:SDS-paf}).

\end{enumerate}

\qed

\subsection{The range $v > 50$} \label{sub:range-2}

In this section we consider the  SDS $(v;r,s;\lambda)$ with 
$v=2n$, where $n=r+s-\lambda$. In particular $v$ is even. 
Recall from Proposition \ref{SDS-kompl} that these SDS give rise to binary periodic complementary pairs, i.e., pairs of $\{ \pm 1 \}$ sequences of length $v$ with PAF constant $\alpha=0$, because of (\ref{jed:SDS-paf}).

Since $52$ is a Golay number, there exist Golay pairs of length 
52. At the same time they are also binary periodic complementary 
pairs.
Since the Diophantine equations $x^2+y^2 = 2 \cdot 54$ and $x^2+y^2 = 2 \cdot 56$ do not have solutions, there are no binary periodic complementary pairs of lengths 54 and 56.
Therefore, the first interesting value of $v > 50 $ is $v = 58$, since this value is not excluded by the restriction of Arasu and Xiang \cite{Arasu:Xiang:DCC:1992}.
In this paper we find $4$ non-equivalent examples of SDS $(58;27,24;22)$, for the first time.
These SDS give $4$ respective pairs of binary sequences of length $58$ with zero periodic autocorrelation function. 

We give four non-equivalent examples of SDS $(58;27,24;22)$:
$$
\begin{array}{l}
  \{ 0,1,2,3,4,7,8,10,11,12,13,16,18,20,24,26,29,31,32,33,36,38,43,46,47,50,53 \} \\
  \{ 0,1,2,3,7,8,10,11,12,13,16,17,21,22,24,27,30,34,41,42,43,45,47,49 \}
\end{array}
$$

$$
\begin{array}{l}
 \{ 0,1,2,3,5,6,7,9,11,12,14,15,17,19,23,24,25,26,29,32,33,39,40,43,45,48,52 \} \\
 \{ 0,1,2,3,4,5,9,11,14,15,16,18,22,26,27,31,32,34,37,39,41,42,45,51 \}
\end{array}
$$

$$
\begin{array}{l}
 \{ 0,1,2,3,5,8,9,11,12,13,14,18,19,21,24,25,27,29,32,34,35,39,41,43,44,48,49 \} \\
 \{ 0,2,3,4,6,8,10,13,16,17,19,20,21,25,28,29,32,33,34,39,40,41,43,46 \}
\end{array}
$$

$$
\begin{array}{l}
 \{ 0,2,3,4,6,7,8,10,11,14,16,17,18,20,23,25,26,28,31,32,36,37,38,41,42,47,49 \} \\
 \{ 0,1,2,3,5,8,9,10,12,16,17,18,22,25,28,30,35,37,41,44,45,46,48,49 \}
\end{array}
$$

All four examples are in the canonical form defined in \cite{Djokovic:AnnComb:2011} and since they are different, this implies that they are non-equivalent.

\section{Acknowledgements}
The authors thank Joe Sawada and Daniel Recoskie for sharing improved
versions of the their C code for computing ordinary and charmed bracelets.
The authors wish to acknowledge generous support by NSERC.
This work was made possible by the facilities of the Shared Hierarchical
Academic Research Computing Network (SHARCNET) and Compute/Calcul Canada.

\end{document}